\documentclass[smallheadings,headsepline,11pt,a4paper]{article}

\usepackage[applemac]{inputenc} 
\usepackage{hyperref} 
\usepackage[english]{babel}
\usepackage{dsfont}
\usepackage{amsmath}
\usepackage{amssymb}
\usepackage{amsthm}
\title{Class Forcing in Class Theory}
\author{Carolin Antos\footnote{I want to thank the Austrian Academy of Sciences for their generous support through their Doctoral Fellowship Program. I would also like to thank my doctoral advisor Sy Friedman for his insightful comments on previous drafts of this paper and his kind support. This paper was prepared as part of the project: ``The Hyperuniverse: Laboratory of the Infinite'' of the JTF grant ID35216.}\\
Kurt Gödel Research Center for Mathematical Logic\\
University of Vienna}

\begin{document}

\newtheorem{defi}{Definition} 
\newtheorem{theo}[defi]{Theorem}
\newtheorem{cor}[defi]{Corollary}
\newtheorem{prop}[defi]{Proposition}
\newtheorem{lem}[defi]{Lemma}
\newtheorem{rem}[defi]{Remark}
\newtheorem{claim}[defi]{Claim}
\newtheorem{fac}{Fact}

\maketitle

\begin{abstract}
In this article we adapt the existing account of class-forcing over a ZFC model to a model $(M,\mathcal{C})$ of Morse-Kelley class theory. We give a rigorous definition of class-forcing in such a model and show that the Definability Lemma (and the Truth Lemma) can be proven without restricting the notion of forcing. Furthermore we show under which conditions the axioms are preserved. We conclude by proving that Laver's Theorem does not hold for class-forcings.
\end{abstract}

\section{Introduction}

The idea of considering a forcing notion with a (proper) class of conditions instead of with a set of conditions was introduced by W. Easton in 1970. He needed the forcing notion to be a class to prove the theorem that the continuum function $2^{\kappa}$, for $\kappa$ regular, can behave in any reasonable way and as changes in the size of $2^{\kappa}$ are bounded by the size of a set-forcing notion, the forcing has to be a class. Two problems arise when considering a class sized forcing: the forcing relation might not be definable in the ground model and the extension might not preserve the axioms. This was addressed in a general way in S.Friedman's book (see \cite{sybook}) where he presented class forcings which are definable (with parameters) over a model $\langle M, A\rangle$. This is called a model of ZF if $M$ is a model of ZF and Replacement holds in $M$ for formulas which mention $A$ as a predicate. We will call such forcings $A$-\emph{definable class-forcings}, their generics $G$ $A$-\emph{definable class-generics} and the resulting new model $A$-\emph{definable class-generic outer model}. Friedman showed that for such $A$-definable class-forcing which satisfy an additional condition called tameness the Definability Lemma, the Truth Lemma and the preservation of the Axioms of ZFC hold. 
In this article we would like to broaden this approach by changing the notion of ground model from a model $M$ of ZFC with a class $A$ to general models of class theory with an arbitrary collection of classes $\mathcal{C}$. We choose Morse-Kelley class theory as our underlying theory and then follow the line of \cite{sybook} and restrict the forcing notion by using the property of tameness. 
\footnote{In \cite{chua} R. Chuaqui follows a similar approach and defines forcing for Morse-Kelley class theory. However there is a significant difference between our two approaches. To show that the extension preserves the axioms Chuaqui restricts the generic $G$ for an arbitrary forcing notion $P$ in the following way: A subclass $G$ of a notion of forcing $P$ is \emph{strongly P-generic over} a model $(M,\mathcal{C})$ of MK iff $G$ is $P$-generic over $(M, \mathcal{C})$ and for all ordinals $\beta\in M$ there is a set $P'\in M$ such that $P'\subseteq P$ and for all sequences of dense sections $\langle D_{\alpha}:\alpha\in\beta\rangle$, there is a $q\in G$ satisfying
\begin{align*}
\forall\alpha(\alpha\in\beta\to\exists\, p\,(p\in P'\cap G\,\wedge\,&\text{the greatest lower bound of $p$ and $q$ exists}\\
&\text{and is an element of } D_{\alpha})).
\end{align*}
where a subclass $D$ of a partial order $P$ is a $P$-\emph{section} if every extension of a condition in $D$ is in $D$.
}

In the following we will introduce Morse-Kelley class theory and define the relevant notions like names, interpretations and the extension for class-forcing in Morse-Kelley. Then we will show that the forcing relation is definable in the ground model, that the Truth Lemma holds and we characterize $P$-generic extensions which satisfy the axioms of MK. We will show that Laver's Theorem fails for class-forcings.


\section{Morse-Kelley Class Theory}\label{axioms}

In ZFC we can only talk about classes by using formulas as our only objects are sets. In class theories like Morse-Kelley (MK) or Gödel-Bernays (GB) the language is two-sorted, i.e. the object are sets and classes and we have corresponding quantifiers for each type of object. \footnote{There is also a one-sorted formulation in which the only objects are classes and sets are defined as being classes which are elements of other classes. For reasons of clarity we will use the two-sorted version throughout the paper.} We denote the classes by upper case letters and sets by lower case letters, the same will hold for class-names and set-names and so on. Hence atomic formulas for the $\in$-relation are of the form ``$x\in X$'' where $x$ is a set-variable and $X$ is a set- or class-variable. The models $\mathcal{M}$ of MK  are of the form $\langle M, \in, \mathcal{C}\rangle$, where $M$ is a transitive model of ZFC, $\mathcal{C}$ the family of classes of $\mathcal{M}$ (i.e. every element of $\mathcal{C}$ is a subset of $M$) and $\in$ is the standard $\in$ relation (from now on we will omit mentioning this relation). 

The axiomatizations of class theories which are often used and closely related to ZFC are MK and GBC. Their axioms which are purely about sets coincide with the corresponding ZFC axioms such as pairing and union and they share class axioms like the Global Choice Axiom. Their difference lies in the Comprehension Axiom in the sense that GB only allows quantification over sets whereas MK allows quantification over sets as well as classes. This results in major differences between the two theories which can be seen for example in their relation to ZFC: GB is a conservative extension of ZFC, meaning that every sentence about sets that can be proved in GB could already be proved in ZFC and so GB cannot prove ``new'' theorems about ZFC. MK on the other hand can do just that, in particular MK implies CON(ZFC) \footnote{This is because in MK we can form a Satisfaction Predicate for $V$ and then by reflection we get an elementary submodel $V_{\alpha}$ of $V$. But any such $V_{\alpha}$ models ZFC.} and so MK is not conservative over ZFC. The consistency strength of MK is strictly stronger than that of ZFC but lies below that of ZFC + an inaccessible cardinal as $\langle V_{\kappa}, V_{\kappa +1}\rangle$ for $\kappa$ inaccessible, is a model for MK in ZFC. 

As said above we choose MK (and not GB) as underlying theory to define class-forcing. The reason lies mainly in the fact that within MK we can show the Definability Lemma for class-forcing without having to restrict the forcing notion whereas in GB this would not be possible. We use the following axiomatization of MK:
\begin{itemize}
\item[A)] Set Axioms:
\begin{enumerate}
\item Extensionality for sets: $\forall x \forall y (\forall z\,(z\in x\leftrightarrow z\in y)\to x=y)$.
\item Pairing: For any sets $x$ and $y$ there is a set $\{x,y\}$.
\item Infinity: There is an infinite set.
\item Union: For every set $x$ the set $\bigcup x$ exists.
\item Power set: For every set $x$ the power set $P(x)$ of $x$ exists.
\end{enumerate}
\item[B)] Class Axioms:
\begin{enumerate}
\item Foundation: Every nonempty class has an $\in$-minimal element.  
\item Extensionality for classes: $\forall z\,(z\in X\leftrightarrow z\in Y)\to X=Y$.
\item Replacement: If a class $F$ is a function and $x$ is a set, then $\{ F(z): z\in x\}$ is a set.
\item Class-Comprehension: \begin{equation*}
\forall X_1\ldots\forall X_n\exists Y\; Y=\{x: \varphi(x,X_1,\ldots, X_n)\}
\end{equation*}
where $\varphi$ is a formula containing class parameters in which quantification over both sets and classes are allowed.
\item Global Choice: There exists a global class well-ordering of the universe of sets.
\end{enumerate}
\end{itemize}

There are different ways of axiomatizing MK, one of them is obtained by using the Limitation of Size Axiom instead of Global Choice and Replacement. Limitation of Size is an axiom that was introduced by von Neumann says that for every $C\in \mathcal{M}$, $C$ is a proper class if and only if there is a one-to-one function from the universe of sets to $C$, i.e. all the proper classes have the same size. The two axiomatizations are equivalent: Global Choice and Replacement follow from Limitation of size and vice versa. \footnote{This is because Global Choice is equivalent with the statement that every proper class is bijective with the ordinals.} A nontrivial argument shows that Limitation of Size does not follow from Replacement plus Local Choice.


In the definition of forcing we will use the following induction and recursion principles:

\begin{prop}[Induction]\label{induction}
Let $(Ord, R)$ be a well-founded relation and $\varphi(\alpha)$ a property of an ordinal $\alpha$. Then it holds that
\begin{equation*}
\forall \alpha\in Ord\, ((\forall \beta\in Ord\,(\beta\, R\,\alpha\to\varphi(\beta)))\to\varphi(\alpha))\to\forall\alpha\in Ord\,\varphi(\alpha)
\end{equation*}
\end{prop}
\begin{proof}
Otherwise, as $R$ is well-founded, there exists an $R$-minimal element $\alpha$ of Ord such that $\neg\varphi(\alpha)$. That is a contradiction.
\end{proof}

\begin{prop}[Recursion]\label{recursion}
For every well-founded binary relation $(Ord, R)$ and every formula $\varphi(X, Y)$ satisfying $\forall X\,\exists !\, Y\,\varphi(X,Y)$, there is a unique binary relation $S$ on $Ord\times V$ such that for every $\alpha\in Ord$ it holds that $\varphi(S_{<\alpha}, S_{\alpha})$ holds, where $S_{\alpha}=\{ x\,\vert\,(\alpha, x)\in S\}$ and $S_{<\alpha}=\{(\beta, x)\in S\,\vert\,\beta R\alpha\}$.
\end{prop}
\begin{proof}
By induction on $\alpha$ it holds that for each $\gamma$ there exists a unique binary relation $S^{\gamma}$ on $Ord_{<\gamma}\times V$, where $Ord_{<\gamma}=\{\beta\in Ord\,\vert\,\beta R\gamma\}$, such that $\varphi(S^{\gamma}_{<\alpha}, S^{\gamma}_\alpha)$ holds for all $\alpha R\gamma$. Then it follows from Class-Comprehension that we can take $S=\bigcup_{\gamma\in Ord} S{\gamma}$.
\end{proof}

\section{Generics, Names and the Extension}
To lay out forcing in MK we have to redefine the basic notions like names, interpretation of names etc. to arrive at the definition of the forcing extension. As we work in a two-sorted theory we will define these notions for sets and classes respectively. Let us start with the definition of the forcing notions and its generics. We use the notation $(X_1,\ldots, X_n)\in \mathcal{C}$ to mean $X_i\in \mathcal{C}$ for all $i$.\\
\begin{defi}
Let $P\in\mathcal{C}$ and $\leq_P\, \in \mathcal{C}$ be a partial ordering with greatest element $1^P$. We call $(P,\leq_P)\in\mathcal{C}$ an $(M,\mathcal{C})$-forcing and often abbreviate it by writing $P$. With the above convention $(P,\leq_P)\in\mathcal{C}$ means that $P$ and $\leq_P$ are in $\mathcal{C}$. \\
$G\subseteq P$ is $P$-generic over $(M,\mathcal{C})$ if 
\begin{enumerate}
\item $G$ is compatible: If $p, q\in G$ then for some $r$, $r\leq p$ and $r\leq q$.
\item $G$ is upwards closed: $p\geq q\in G\,\to\, p\in G$.
\item $G\cap D\neq \emptyset$ whenever $D\subseteq P$ is dense, $D\in\mathcal{C}$.
\end{enumerate}
\end{defi}
Note that from now on we will assume $M$ to be countable (and transitive) and $\mathcal{C}$ to be countable to ensure that for each $p\in P$ there exists $G$ such that $p\in G$ and $G$ is $P$-generic.

We build the hierarchy of names for sets and classes in the following way (we will use capital greek letters for class-names and lower case greek letters for set-names): 
\begin{defi}\label{names}
\begin{itemize}
\item[]
\item[] $\mathcal{N}^s_0=\emptyset$.
\item[] $\mathcal{N}^s_{\alpha+1}=\{\sigma : \sigma\text{ is a subset of }\mathcal{N}^s_{\alpha}\times P \text{ in } M\}$.
\item[] $\mathcal{N}^s_{\lambda}=\bigcup\{\mathcal{N}^s_{\alpha}:\alpha<\lambda\}$, if $\lambda$ is a limit ordinal.
\item[] $\mathcal{N}^s=\bigcup\{\mathcal{N}^s_{\alpha}:\alpha\in ORD(M)\}$ is the class of all set-names of P. 
\item[] $\mathcal{N}=\{\Sigma: \Sigma\text{ is a subclass of }\mathcal{N}^s\times P\text{ in } \mathcal{C}\}$.
\end{itemize}
\end{defi}

Note that the $\mathcal{N}^s_{\alpha}$ (for $\alpha>0$) are in fact proper classes and therefore Definition \ref{names} is an inductive definition of a sequence of proper classes of length the ordinals. The fact that with this definition we stay inside $\mathcal{C}$ follows from Proposition \ref{recursion}.

\begin{lem}
\begin{itemize}
\item[]
\item[a)] If $\alpha\leq\beta$ then $\mathcal{N}^s_{\alpha}\subseteq\mathcal{N}^s_{\beta}$.
\item[b)] $\mathcal{N}^s\subseteq\mathcal{N}$.
\end{itemize}
\end{lem}
\begin{proof} 
\begin{itemize}
\item [a)] By induction on $\beta$. For $\beta=0$ there is nothing to prove.

Successor step $\beta\to\beta+1$. Assume $\mathcal{N}^s_{\alpha}\subseteq\mathcal{N}^s_{\beta}$ for all $\alpha\leq\beta$. Let $\tau\in\mathcal{N}^s_{\alpha}$ for some $\alpha<\beta+1$. Then we know by assumption that $\tau\in\mathcal{N}^s_{\beta}$. So by Definition \ref{names}  there is some $\gamma<\beta$ such that $\tau=\{\langle\pi_i, p_i\rangle\,\vert\, i\in I\}$ where for each $i\in I$, $\pi_i\in\mathcal{N}^s_{\gamma}$ and $p_i\in P$. By assumption $\pi_i\in\mathcal{N}^s_{\beta}$ for all $i\in I$ and so $\tau\in\mathcal{N}^s_{\beta+1}$. 

Limit step $\lambda$. Assume $\mathcal{N}^s_{\alpha}\subseteq\mathcal{N}^s_{\beta}$ for all $\alpha\leq\beta<\lambda$. But by Definition \ref{names}, $\sigma\in\mathcal{N}^s_{\lambda}$ iff $\sigma\in\mathcal{N}^s_{\beta}$ for some $\beta<\lambda$ and so it follows that $\mathcal{N}^s_{\alpha}\subseteq \mathcal{N}^s_{\lambda}$ for all $\alpha\leq\lambda$.

\item[b)] By Definition \ref{names}, $\Sigma\in\mathcal{N}$ iff $\Sigma$ is a subclass of $\mathcal{N}^s\times P$ iff for every $\langle\tau, p\rangle\in\Sigma$, $\tau\in\mathcal{N}^s$ and $p\in P$ iff for every $\langle\tau, p\rangle\in\Sigma$ there is an ordinal $\alpha$ such that $\tau\in\mathcal{N}^s_{\alpha}$ and $p\in P$. Let $\sigma\in\mathcal{N}^s$, i.e. there is an ordinal $\beta$ such that $\sigma\in\mathcal{N}^s_{\beta}$. Then it holds that for every $\langle\tau, p\rangle\in\sigma$ there is an ordinal $\alpha<\beta$ such that $\tau\in\mathcal{N}^s_{\alpha}$ and $p\in P$. So $\sigma\in\mathcal{N}$.
\end{itemize}
\end{proof}
We define the interpretations of set- and class-names recursively.
\begin{defi}\label{interpretation}
\begin{itemize}
\item[]
\item[] $\sigma^G=\{\tau^{G}:\exists p\in G(\langle\tau, p\rangle\in\sigma)\}$ for $\sigma\in\mathcal{N}^s$.
\item[] $\Sigma^G=\{\sigma^{G}:\exists p\in G(\langle\sigma, p\rangle\in\Sigma)\}$ for $\Sigma\in\mathcal{N}$.
\end{itemize}
\end{defi}

According to the definitions above we define the extension of an MK model $(M, \mathcal{C})$ to be the extension of the set part and the extension of the class part:
\begin{defi}\label{extension}
$(M,\mathcal{C})[G]= (M[G],\mathcal{C}[G])=(\{\sigma^G:\sigma\in\mathcal{N}^s\},\{\Sigma^{G}:\Sigma\in\mathcal{N}\})$.
\end{defi}
\begin{defi}\label{canonicalnames}
If $P$ is a partial order with greatest element $1^P$, we define the canonical $P$-names of $x\in M$ and $C\in\mathcal{C}$:
\begin{itemize}
\item[] $\check{x}=\{\langle\check{y}, 1^P\rangle\,\vert\, y\in x\}$.
\item[] $\check{C}=\{\langle\check{x}, 1^P\rangle\,\vert\, x\in C\}$.
\end{itemize}
\end{defi}
From these definitions the basic facts of forcing follow easily:
\begin{lem}  \label{basics}
Let $\mathcal{M}=\langle M, \mathcal{C}\rangle$ be a model of MK, where $M$ is a transitive model of ZFC and $\mathcal{C}$ the family of classes of $\mathcal{M}$. Then it holds that:
\begin{itemize} 
\item[a)] $\forall x\in M(\check{x}\in \mathcal{N}^s\wedge \check{x}^G=x)$ and $\forall C\in\mathcal{C}(\check{C}\in\mathcal{N}\wedge\check{C}^{G}=C)$.
\item[b)] $(M,\mathcal{C})\subseteq(M,\mathcal{C})[G]$ in the sense that $M\subseteq M[G]$ and $\mathcal{C}\subseteq\mathcal{C}[G]$.
\item[c)] $G\in(M,\mathcal{C})[G]$, i.e. $G\in \mathcal{C}[G]$
\item[d)] $M[G]$ is transitive and $Ord(M[G]))=Ord(M)$. 
\item[e)] If $( N, \mathcal{C}')$ is a model of MK, $M\subseteq N$, $\mathcal{C}\subseteq\mathcal{C}'$, $G\in\mathcal{C}'$ then $(M,\mathcal{C})[G]\subseteq(N, \mathcal{C}')$. 

\end{itemize}
\end{lem}
\begin{proof}
a): Using Definition \ref{interpretation} and Definition \ref{canonicalnames} we can easily show this by induction.
 
 b): follows immediately from 1.
 
 c): Let $\Gamma=\{\langle\check{p},p\rangle:p\in P\}$. Then this is a name for $G$ as $\Gamma^G=\{\check{p}^{G}\,\vert\,p\in G\}=\{p\,\vert\,p\in G\}=G$. 
 
 d) It follows from Definition \ref{interpretation} and Definition \ref{extension} that $M[G]$ is transitive. For every $\sigma\in N^s$ 
  the rank of $\sigma^{G}$ is at most rank $\sigma$, so $Ord(M)[G])\subseteq Ord(M)$. 
  
  e) For each name $\Sigma\in \mathcal{N}$, $\Sigma\in(M, \mathcal{C})$ and therefore $\Sigma\in (N, \mathcal{C}')$. As $G\in \mathcal{C}'$ the interpretation of $\Sigma$ in $(M,\mathcal{C})[G]$ is the same as in $(N,\mathcal{C}')$.
  
\end{proof}
\section{Definability and Truth Lemmas}
We will define the forcing relation and show that it is definable in the ground model and how it relates to truth in the extension. The main focus will be the Definability Lemma, since it now is possible to prove that it holds for all forcing notions in contrast to $A$-definable class-forcings in a ZFC setting. Note that when we talk about a formula $\varphi(x_1,\ldots,x_m,X_1,\ldots,X_n)$ we mean $\varphi$ to be a second-order formula that allows second-order quantification and we always assume the model $(M, \mathcal{C})$ to be countable.
\begin{defi}
Suppose $p$ belongs to $P$, $\varphi(x_1,\ldots,x_m,X_1,\ldots,X_n)$ is a formula, $\sigma_1,\ldots,\sigma_m$ are set-names and $\Sigma_1,\ldots,\Sigma_n$ are class-names. We write $p\Vdash\varphi(\sigma_1,\ldots,\sigma_m,\Sigma_1,\ldots,\Sigma_n)$ iff whenever $G\subseteq P$ is $P$-generic over $(M,\mathcal{C})$ and $p\in P$, we have $(M,\mathcal{C})[G]\models \varphi(\sigma^G_1,\ldots,\sigma^G_m,\Sigma^G_1,\ldots,\Sigma^G_n)$.
\end{defi}
\begin{lem}[Definability Lemma]\label{definability}
For any $\varphi$, the relation ``$p\Vdash\varphi(\sigma_1,\ldots,\sigma_m,\Sigma_1,\ldots,\Sigma_n)$'' of $p$, $\vec{\sigma}$, $\vec{\Sigma}$ is definable in $(M,\mathcal{C})$.
\end{lem}
\begin{lem}[Truth Lemma]\label{truthlemma}
If $G$ is $P$-generic over $(M, \mathcal{C})$ then 
\begin{equation*}
(M,\mathcal{C})[G]\models\varphi(\sigma^G_1,\ldots,\sigma^G_m,\Sigma^G_1,\ldots,\Sigma^G_n)\Leftrightarrow\exists p\in G\,(p\Vdash\varphi(\sigma_1,\ldots,\sigma_m,\Sigma_1,\ldots,\Sigma_n)).
\end{equation*}
\end{lem}

Following the approach of set-forcing we introduce a new relation $\Vdash^*$ and prove the Definability and Truth Lemma for this $\Vdash^*$. Then we will show that $\Vdash^*$ equals the intended forcing relation $\Vdash$.

The definition of $\Vdash^*$ consists of ten cases: six cases for atomic formulas, where the first two are for set-names, the second two for the ``hybrid'' of set- and class-names and the last two for class-names, one for $\wedge$ and $\neg$ respectively and two quantifier cases, one for first-order and one for second-order quantification. Of course formally there are only two cases for atomic formulas, namely case five and six of Definition \ref{forcing}. But by splitting the cases we can see very easily that it is enough to prove the Definability Lemma for set-names only (case one and two in the Definition) and then infer the general Definability Lemma by induction.

\begin{defi} 
$D\subseteq P$ is \emph{dense below} $p$ if $\forall q≤p\,\exists r\,(r≤q, r\in D)$.
\end{defi}
\begin{defi}\label{forcing}
Let $\sigma, \gamma, \pi$ be elements of $\mathcal{N}^s$ and $\Sigma, \Gamma$ elements of $\mathcal{N}$.  
\begin{enumerate}
\item $p\Vdash^*\sigma\in\gamma$ iff $\{q :\exists\langle\pi, r\rangle\in \gamma\text{ such that }q\leq r, q\Vdash^*\pi=\sigma\}$ is dense below $p$.
\item $p\Vdash^*\sigma=\gamma$ iff for all $\langle\pi, r\rangle\in\sigma\cup\gamma,\; p\Vdash^*(\pi\in\sigma\leftrightarrow\pi\in\gamma)$.
\item $p\Vdash^* \sigma\in\Sigma$ iff $\{q :\exists\langle\pi, r\rangle\in \Sigma\text{ such that }q\leq r, q\Vdash^*\pi=\sigma\}$ is dense below $p$.
\item $p\Vdash^*\sigma=\Sigma$ iff for all $\langle\pi, r\rangle\in\sigma\cup\Sigma,\; p\Vdash^*(\pi\in\sigma\leftrightarrow\pi\in\Sigma)$.
\item $p\Vdash^*\Sigma\in\Gamma$ iff $\{q :\exists\langle\pi, r\rangle\in \Gamma\text{ such that }q\leq r, q\Vdash^*\pi=\Sigma\}$ is dense below $p$.
\item $p\Vdash^*\Sigma=\Gamma$ iff for all $\langle\pi, r\rangle\in\Sigma\cup\Gamma,\; p\Vdash^*(\pi\in\Sigma\leftrightarrow\pi\in\Gamma)$.

\item $p\Vdash^*\varphi\wedge\psi$ iff $p\Vdash^*\varphi$ and $p\Vdash^*\psi$.
\item $p\Vdash^*\neg\varphi$ iff $\forall q≤p\,(\neg\, q\Vdash^*\varphi)$.
\item $p\Vdash^*\forall x\varphi$ iff for all $\sigma$, $p\Vdash^*\varphi(\sigma)$.
\item $p\Vdash^*\forall X\varphi$ iff for all $\Sigma$, $p\Vdash^*\varphi(\Sigma)$.
\end{enumerate}
\end{defi}

We have to show that $\Vdash^*$ is definable within the ground model. For this it is enough to concentrate on the first two of the above cases, because we can reduce the definability of the $\Vdash^*$-relation for arbitrary second-order formulas to its definability for atomic formulas $\sigma\in\tau$, $\sigma=\tau$, where $\sigma$ and $\tau$ are set-names. The rest of the cases then follow by induction. 
So let us restate Lemma \ref{truthlemma} for the case of $\Vdash^*$ and set-names:
\begin{lem}[Definability Lemma for the atomic cases of set-names]\label{definabilityset}
The relation ``$p\Vdash^*\varphi(\sigma, \tau)$'' is definable in $(M, \mathcal{C})$ for $\varphi= ``\sigma\in \tau$'' and $\varphi=``\sigma=\tau$''.
\end{lem}


\begin{proof}
We will show by induction\footnote{To show how this induction works in the context of a class-theory we will not simply use Proposition \ref{induction} and \ref{recursion}, but rather give the complete construction.} on $\beta\in ORD$ that there are unique classes $X_{\beta}, Y_{\beta}\subseteq\beta\times M$ which define the $\Vdash^*$-relation for the first two cases of Definition \ref{forcing} in the following way: for all $\alpha<\beta$, $R_{\alpha}=(X_{\beta})_{\alpha}, S_{\alpha}=(Y_{\beta})_{\alpha}$ where $(X_{\beta})_{\alpha}=\{x\,\vert\,\langle\alpha,x\rangle\in X_{\beta}\}$ and
\begin{align*}
\tag{$\star$} R_{\alpha}=\{(p,\sigma,\in,\tau)\,\vert\, &p\in P, \sigma\text{ and }\tau\text{ are set }P\text{-names,}\\
&\text{rank}(\sigma)\text{ and rank}(\tau)<\alpha,\text{ for all } q\leq p\\
&\text{there is } q'\leq q \text{ and } \langle\pi,r\rangle\in\tau \text{ such that }\\
&q'\leq r\text{ and } (q',\pi,=,\sigma)\in S_{\alpha}\}
\end{align*}
and

\begin{align*}
\tag{$\star\star$} S_{\alpha}=\{(p,\sigma,=,\tau)\,\vert\, &p\in P, \sigma\text{ and }\tau\text{ are set }P\text{-names,}\\
&\text{rank}(\sigma)\text{ and rank}(\tau)<\alpha,\\
&\text{for all }\langle\pi,r\rangle\in\sigma\cup\tau\text{ such that }\\
&(p,\pi,\in,\sigma)\in R_{\alpha}\text{ iff }(p,\pi,\in,\tau)\in R_{\alpha}\}
\end{align*}

To show that $X_{\beta}$ and $Y_{\beta}$ are definable we will define the classes $R_{\alpha}$ and $S_{\alpha}$ at each step by recursion on the tupel $(p,\sigma, e,\tau)$ according to the following well-founded partial order on $P\times \mathcal{N}^s\times\{\text{``}\in\text{''},\text{``}=\text{''}\}\times\mathcal{N}^s$. 
\begin{defi}\label{andrew}
Suppose $(p,\sigma, e, \tau), (q, \sigma',e',\tau')\in P\times \mathcal{N}^s\times\{\text{``}\in\text{''},\text{``}=\text{''}\}\times\mathcal{N}^s$. Say that $(q,\sigma', e', \tau')<(p,\sigma, e, \tau)$if 
\begin{itemize}
\item $max(rank(\sigma',\tau')) <max(rank(\sigma,\tau))$, or
\item $max(rank(\sigma',\tau')) = max(rank(\sigma,\tau))$, and $rank(\sigma)$ $\geq$ $rank(\tau)$ but $rank(\sigma') < rank(\tau')$, or
\item $max(rank(\sigma',\tau')) = max(rank(\sigma,\tau))$, and $rank(\sigma)$ $\geq$ $rank(\tau)$ $\leftrightarrow$ $rank(\sigma')$ $\geq$ $rank(\tau')$, and $e$ is ``$=$'' and $e'$ is ``$\in$''.
\end{itemize}
\end{defi}
Note that clause 1 and 2 of Definition \ref{forcing} always reduce the $<$-rank of the members of $P\times \mathcal{N}^s\times\{\text{``}\in\text{''},\text{``}=\text{''}\}\times\mathcal{N}^s$.\\

``Successor step $\beta\to\beta+1$.'' We know that there are unique classes $X_{\beta}, Y_{\beta}$ such that for all $\alpha<\beta$, $R_{\alpha}=(X_{\beta})_{\alpha}, S_{\alpha}=(Y_{\beta})_{\alpha}$ and ($\star$) and ($\star\star$) hold. We want to show that there are unique classes $X_{\beta+1}, Y_{\beta+1}$ such that for all $\alpha<\beta+1$, $R_{\alpha}=(X_{\beta+1})_{\alpha}, S_{\alpha}=(Y_{\beta+1})_{\alpha}$ and ($\star$) and ($\star\star$) hold. So let for all $\alpha<\beta$ $(X_{\beta+1})_{\alpha}=(X_{\beta})_{\alpha}=R_{\alpha}$ and $(Y_{\beta+1})_{\alpha}=(Y_{\beta})_{\alpha}=S_{\alpha}$ and define $(X_{\beta+1})_{\beta}=R_{\beta}$ and $(Y_{\beta+1})_{\beta}=S_{\beta}$ uniquely as follows:
\begin{itemize}
\item[A)] $(p,\sigma,$``$\in$''$, \tau)\in R_{\beta}$ if and only if for all $q\leq p$ there is $q'\leq q$ and $\langle\pi, r\rangle\in\tau$ such that $q'\leq r$ and $(q',\pi,$``=''$, \sigma)\in S_{\beta}$.
\item[B)] $(p,\sigma,$``$=$''$, \tau)\in S_{\beta}$ if and only if for all $\langle\pi, r\rangle\in\sigma\cup\tau$: $(p,\pi,$``$\in$''$,\sigma)\in R_{\beta}\text{ iff }(p,\pi,$``$\in$''$,\tau)\in R_{\beta}$. 
\end{itemize}
These definitions clearly satisfy $(\star)$ and $(\star\star)$ and to see that they are indeed inductive definitions over the well-order defined in Definition \ref{andrew}, we consider the following three cases for each of the definitions A) and B):
\begin{enumerate}
\item rank$(\sigma)<\text{rank}(\tau)$
\item $\text{rank}(\tau)<\text{rank}(\sigma)$
\item rank$(\sigma)=\text{rank}(\tau)$
\end{enumerate}

Ad A.1: $(q',\pi,$``=''$, \sigma)<(p,\sigma,$``$\in$''$, \tau)$ because $\text{rank}(\sigma),\text{rank}(\pi)<\text{rank}(\tau)$ (first clause of Denfition \ref{andrew}).

Ad A.2: $(q',\pi,$``=''$, \sigma)<(p,\sigma,$``$\in$''$, \tau)$ because $\text{max(rank}(\pi),\text{rank}(\sigma))=\text{max(rank}(\sigma),\text{rank}(\tau))$ and $\text{rank}(\sigma)\geq\text{rank}(\tau)$ and $\text{rank}(\pi)<\text{rank}(\sigma)$ (second clause of Definition \ref{andrew}).

Ad A.3: $(q',\pi,$``=''$, \sigma)<(p,\sigma,$``$\in$''$, \tau)$ because $\text{max(rank}(\pi),\text{rank}(\sigma))=\text{max(rank}(\sigma),\text{rank}(\tau))$ and $\text{rank}(\sigma)\geq\text{rank}(\tau)$ and $\text{rank}(\pi)<\text{rank}(\sigma)=\text{rank}(\tau)$ (second clause of Definition \ref{andrew}).\\

Ad B.1: $(p,\pi,$``$\in$''$,\sigma)<(p,\sigma,$``$=$''$, \tau)$ because $\text{rank}(\sigma),\text{rank}(\pi)<\text{rank}(\tau)$ and $(p,\pi,$``$\in$''$,\tau)<(p,\sigma,$``$=$''$, \tau)$ because $\text{max(rank}(\pi),\text{rank}(\tau))=\text{max(rank}(\sigma),\text{rank}(\tau))$ and $\text{rank}(\sigma)<\text{rank}(\tau)$ and $\text{rank}(\pi)<\text{rank}(\tau)$ (third clause of Definition \ref{andrew}).

Ad B.2: $(p,\pi,$``$\in$''$,\sigma)<(p,\sigma,$``$=$''$, \tau)$ because of the second clause of Definition \ref{andrew} and $(p,\pi,$``$\in$''$,\tau)< (p,\sigma,$``$=$''$, \tau)$ because $\text{rank}(\pi),\text{rank}(\tau)<\text{rank}(\sigma)$.

Ad B.3: $(p,\pi,$``$\in$''$,\sigma)<(p,\sigma,$``$=$''$, \tau)$ and $(p,\pi,$``$\in$''$,\tau)< (p,\sigma,$``$=$''$, \tau)$ because $\text{max(rank}(\pi),\text{rank}(\tau))=\text{max(rank}(\sigma),\text{rank}(\tau))$ and $\text{rank}(\sigma)\geq\text{rank}(\tau)$ and $\text{rank}(\pi)<\text{rank}(\sigma),\text{rank}(\tau)$ (both second clause of Definition \ref{andrew}). \\

``Limit step $\lambda$.'' We know that for every $\beta<\lambda$ there are unique classes $X_{\beta}, Y_{\beta}$ such that for all $\alpha<\beta$, $R_{\alpha}=(X_{\beta})_{\alpha}, S_{\alpha}=(Y_{\beta})_{\alpha}$ and ($\star$) and ($\star\star$) hold. We have to show that there are unique classes $X_{\lambda}, Y_{\lambda}\subseteq\lambda\times M$, $\lambda$ limit, such that for all $\beta<\lambda$, $R_{\beta}=(X_{\lambda})_{\beta}, S_{\beta}=(Y_{\lambda})_{\beta}$ and $(\star)$ and $(\star\star)$ hold respectively. We define the required classes as follows:
\begin{align*}
\langle\alpha, x\rangle\in X_{\lambda}\leftrightarrow\,&\exists\langle\langle R_{\gamma}, S_{\gamma}\rangle\,\vert\,\gamma\leq\alpha\rangle\,\exists X,Y((\forall\gamma\leq\alpha ((X)_{\gamma}=R_{\gamma}\text{ and }\\
&(Y)_{\gamma}=S_{\gamma}\text{ and they satisfy $(\star)$ and $(\star\star)$ resp.})\,\wedge\\
&(x\in(X)_{\gamma}\text{ for some }\gamma\leq\alpha))
\end{align*}
\begin{align*}
\langle\alpha, x\rangle\in Y_{\lambda}\leftrightarrow&\,\exists\langle\langle R_{\gamma}, S_{\gamma}\rangle\,\vert\,\gamma\leq\alpha\rangle\,\exists X,Y((\forall\gamma\leq\alpha ((X)_{\gamma}=R_{\gamma}\text{ and }\\
&(Y)_{\gamma}=S_{\gamma}\text{ and they satisfy $(\star)$ and $(\star\star)$ resp.})\,\wedge\\
&(x\in(Y)_{\gamma}\text{ for some }\gamma\leq\alpha))
\end{align*}
From the proof of the successor step we see that the sequence $\langle\langle R_{\gamma}, S_{\gamma}\rangle\,\vert\,\gamma\leq\alpha\rangle$ is unique for every $\alpha<\lambda$ and therefore $X_{\lambda}, Y_{\lambda}$ are also unique.
This definition is possible only in Morse-Kelly with its version of Class-Comprehension and not in Gödel-Bernays, because we are quantifying over class variables (in fact we only need $\Delta^1_1$ Class-Comprehension).
\end{proof}
The general Definability Lemma now follows immediately from this Lemma and Definition \ref{forcing}. We now turn to the Truth Lemma.  

In the following a capital greek letter denotes a name from $\mathcal{N}$ (and therefore can be a set- or a class-name), whereas a lower case greek letter is a name from $\mathcal{N}^s$ (and therefore can only be a set-name).
\begin{lem}\label{auxlemma} 
\begin{itemize}
\item[]
\item[a)] If $p\Vdash^*\varphi$ and $q\leq p$ then $q\Vdash^*\varphi$
\item[b)] If $\{ p\,\vert\, q\Vdash^*\varphi\}$ is dense below $p$ then $p\Vdash^*\varphi$.
\item[c)] If $\neg p\Vdash^*\varphi$ then $\exists q\leq p(q\Vdash^*\neg\varphi)$.
\end{itemize}
\end{lem}
\begin{proof}
a) By induction on $\varphi$: Let $\varphi$ be $\Sigma\in\Gamma$, then by Definition \ref{names} $D=\{q' :\exists\langle\pi, r\rangle\in \Gamma\text{ such that }q'\leq r, q'\Vdash^*\pi=\Sigma\}$ is dense below $p$. Then for all $q\leq p$, $D$ is also dense below $q$ and therefore $q\Vdash^*\varphi$. The other cases follow easily.

b) By induction on $\varphi$. Let $\varphi$ be $\Sigma\in\Gamma$ and $\{ q\,\vert\, q\Vdash^*\Sigma\in\Gamma\}$ is dense below $p$. From Definition \ref{forcing} it follows that $\{ q\,\vert\, \{s: \exists\langle\pi, r\rangle\in\Gamma\text{ such that }s\leq r, s\Vdash^*\pi=\Sigma\}\text{ is dense below }q\}$ is dense below $p$ and from a well-known fact it follows that $D=\{s: \exists\langle\pi, r\rangle\in\Gamma\text{ such that }s\leq r, s\Vdash^*\pi=\Sigma\}$ is dense below $p$. Again by Definition \ref{forcing} we get as desired $p\Vdash^* \Sigma\in\Gamma$.

The other cases follow easily; for the case of negation we will use the fact that if $\{p\,\vert\, q\Vdash^*\neg\varphi\}$ is dense below $p$ then $\forall q\leq p (\neg q\Vdash^*\varphi)$, using a).

c) follows directly from b).

\end{proof}
Now, the proofs for the Truth Lemma and $\Vdash^* = \Vdash$ follow similarly to the proofs in set-forcing (note that a name $\Sigma\in\mathcal{N}$ can also be a set-name and therefore we don't need to mention the cases for set-names explicitly):
\begin{lem}[Truth Lemma]\label{truthlemmaset}
If $G$ is $P$-generic then 
\begin{equation*}
(M,\mathcal{C})[G]\models\varphi(\Sigma^G_1,\ldots,\Sigma^G_m)\Leftrightarrow\exists p\in G\,(p\Vdash^*\varphi(\Sigma_1,\ldots,\Sigma_m)).
\end{equation*}
\end{lem}
\begin{proof}
By induction on $\varphi$.
\begin{itemize}
\item[$\Sigma\in\Gamma$.] ``$\rightarrow$'' Assume $\Sigma^G\in \Gamma^G$ then choose a $\langle\pi, r\rangle\in\Gamma$ such that $\Sigma^G=\pi^G$ and $r\in G$. By induction there is a $p\in G$ with $p\leq r$ and $p\Vdash^*\pi=\Sigma$. Then for all $q\leq p$, $q\Vdash^*\pi=\Sigma$ and by Definition \ref{names} $p\Vdash^*\Sigma\in\Gamma$.

``$\leftarrow$'': Assume $\exists p\in G (p\Vdash^*\Sigma\in\Gamma)$. Then $\{q :\exists\langle\pi, r\rangle\in \tau\text{ such that }q\leq r, q\Vdash^*\sigma=\pi\}=D$ is dense below $p$ and so by genericity $G\cap D\neq\emptyset$. So there is a $q\in G$, $q\leq p$ such that $\exists \langle\pi, r\rangle\in\Gamma$ with $q\leq r$, $q\Vdash^*\pi=\Sigma$. By induction $\pi^{G}=\Sigma^{G}$ and as $r\geq q$, $r\in G$  and therefore $\pi^{G}\in\Gamma^{G}$. So $\Sigma^{G}\in\Gamma^{G}$.

\item[$\Sigma=\Gamma$.] ``$\rightarrow$'' Assume $\sigma^{G}=\Gamma^{G}$. Then for all $\langle\pi, r\rangle\in\Sigma\cup\Gamma$ with $r\in G$ it holds that $\pi^{G}\in\Sigma^{G}\leftrightarrow\pi^{G}\in\Gamma^{G}$. Let $D=\{p\,\vert\,\text{either }p\Vdash^*\Sigma=\Gamma\text{ or for some }\langle\pi, r\rangle\in\Sigma\cup\Gamma, p\Vdash^*\neg(\pi\in\Sigma\leftrightarrow\pi\in\Gamma)\}$. Then $D$ is dense: By contradiction, let $q\in P$ and assume that there is no $p\leq q$ such that $p\in D$. 
But if there is no $p\leq q$ such that for some $\langle\pi, r\rangle\in\Sigma\cup\Gamma, p\Vdash^*\neg(\pi\in\Sigma\leftrightarrow\pi\in\Gamma)\}$ then by Lemma \ref{auxlemma} $q\Vdash^*(\pi\in\Sigma\leftrightarrow\pi\in\Gamma)$ for all $\langle\pi, r\rangle\in\Sigma\cup\Gamma$ and therefore $q\Vdash^*\Sigma=\Gamma$. So there is a $p\leq q$ such that $p\in D$. Since the filter $G$ is generic, there is a $p\in G\cap D$. If $p\Vdash^*\neg(\pi\in\Sigma\leftrightarrow\pi\in\Gamma)\}$ for some $\langle\pi, r\rangle\in\Sigma\cup\Gamma$ then by induction $\neg(\pi^{G}\in\Sigma^{G}\leftrightarrow\pi^{G}\in\Gamma^{G})$ for some $\langle\pi, r\rangle\in\Sigma\cup\Gamma$. But this is a contradiction to $\Sigma^{G}=\Gamma^{G}$ and so $P\Vdash^*\Sigma=\Gamma$.

``$\leftarrow$'' Assume that there is $p\in G$ $(p\Vdash^*\Sigma=\Gamma)$. By Definition \ref{names} it follows that for all $\langle\pi, r\rangle\in\Sigma\cup\Gamma$ $P\Vdash^*(\pi\in\Sigma\leftrightarrow\pi\in\Gamma)$. Then by induction $\pi^{G}\in\Sigma^{G}\leftrightarrow\pi^{G}\in\Gamma^{G}$ for all $\langle\pi, r\rangle\in\Sigma\cup\Gamma$. So $\Sigma^{G}=\Gamma^{G}$.

\item[$\varphi\wedge\psi$] ``$\rightarrow$'' Assume that $(M, \mathcal{C})[G]\models\varphi\wedge\psi$ iff $(M, \mathcal{C})[G]\models\varphi$ and $(M, \mathcal{C})[G]\models\psi$. Then by induction $\exists p\in G$ $P\Vdash^*\varphi$ and $\exists q\in G$, $q\Vdash^*\psi$ and we know that $\exists r\in G (r\leq p\text{ and }r\leq q)$ such that $r\Vdash^*\varphi$ and $r\Vdash^*\psi$ and so by Definition \ref{names} $r\Vdash^*\varphi\wedge\psi$.

``$\leftarrow$''  Assume $\exists p\in G, p\Vdash^*\varphi\wedge\psi$, then $p\Vdash^*\varphi$ and $p\Vdash^*\psi$. So $(M, \mathcal{C})[G]\models\varphi$ and $(M, \mathcal{C})[G]\models\psi$ and therefore $(M, \mathcal{C})[G]\models\varphi\wedge\psi$.

\item[$\neg\varphi$] ``$\rightarrow$'' Assume that $(M, \mathcal{C})[G]\models\neg\varphi$. $D=\{ p\,\vert\, p\Vdash^*\varphi$ or $p\Vdash^*\neg\varphi\}$ is dense (using Lemma \ref{auxlemma} and Definition \ref{names}). Therefore there is a $p\in G\cap D$ and by induction $p\Vdash^*\neg\varphi$.

``$\leftarrow$'' Assume that there is $p\in G$ such that $p\Vdash^*\neg\varphi$. If $(M, \mathcal{C})\models \varphi$ then by induction hypothesis there is a $q\in G$ such that $q\Vdash^*\varphi$. But then also $r\Vdash^*\varphi$ for some $r\leq p, q$ and this is a contradiction because of Definition \ref{names}. So $(M, \mathcal{C})\models\neg\varphi$.

\item[$\forall X\varphi$] ``$\rightarrow$'' Assume that $(M, \mathcal{C})[G]\models\forall X\varphi$. Following the lines of the ``$\rightarrow$''-part of the proof for $\Sigma=\Gamma$, there is a dense $D=\{p\,\vert\,\text{either }p\Vdash^*\forall X\varphi\text{ or for some }\sigma, p\Vdash^*\neg\varphi(\sigma)\}$. By induction we show that the second case is not possible and so it follows that $p\Vdash^*\forall X\varphi$.

``$\leftarrow$'' By induction.
\end{itemize}
\end{proof}
\begin{lem}
$\Vdash^*=\Vdash$
\end{lem}
\begin{proof} 
$p\Vdash^*\varphi(\sigma_1, \ldots, \sigma_n)\rightarrow p\Vdash\varphi(\sigma_1, \ldots, \sigma_n)$ follows directly from the Truth Lemma.
For the converse we use Lemma \ref{auxlemma} c) and note that we assumed the existence of generics. Then from $\neg p\Vdash^*\varphi(\sigma_1, \ldots, \sigma_n)$ it follows that for some $q\leq p$, $q\Vdash^*\neg \varphi(\sigma_1, \ldots, \sigma_n)$ and so $\neg p\vdash \varphi(\sigma_1, \ldots, \sigma_n)$.
\end{proof}

\section{The Extension fulfills the Axioms}




We have shown that in MK we can prove the Definability Lemma without restricting the forcing notion as we have to do when working with $A$-definable class-forcing in ZFC. Unfortunately we do not have the same advantage when proving the preservation of the axioms. For example, when proving the Replacement Axiom we have to show that the range of a set under a class function is still a set and this does not hold in general for class-forcings. In \cite{sybook} two properties of forcing notions are introduced, namely pretameness and tameness. Pretameness is needed to prove the Definability Lemma and show that all axioms except Power Set are preserved. For the Power Set Axiom this restriction needs to be strengthened to tameness. Let us give the definitions in the MK context:
\begin{defi}[Pretameness]
$D\subseteq P$ is predense $\leq$ $p\in P$ if every $q\leq p$ is compatible with an element of $D$.\\
$P$ is pretame if and only if whenever $\langle D_i\,\vert\, i\in a\rangle$ is a sequence of dense classes in $\mathcal{M}$, $a\in M$ and $p\in P$ then there exists a $q\leq p$ and $\langle d_i\,\vert\, i\in a\rangle \in M$ such that $d_i\subseteq D_i$ and $d_i$ is predense $\leq q$ for each $i$.
\end{defi}
\begin{defi}
$q\in P$ meets $D\subseteq P$ if $q$ extends an element in $D$.

A predense $\leq p$ partition is a pair $(D_0,D_1)$ such that $D_0\cup D_1$ is predense $\leq p$ and $p_0\in D_0, p_1\in D_1\rightarrow p_0, p_1$ are incompatible. Suppose $\langle (D_0^i, D_1^i)\,\vert\, i\in a\rangle$, $\langle (E^i_0, E^i_1)\,\vert\, i\in a\rangle$ are sequences of predense $\leq p$ partitions. We say that they are equivalent $\leq p$ if for each $i\in a$, $\{ q \,\vert\, q\text{ meets }D^i_0\leftrightarrow q\text{ meets }E^i_0\}$ is dense $\leq p$. When $p=1^P$ we omit $\leq p$.
\end{defi}

To each sequence of predense $\leq p$ partitions $\vec{D}=\langle(D^i_0, D^i_1)\vert i\in a\rangle\in M$ and $G$ is $P$-generic over $\langle M,\mathcal{C}\rangle$, $p\in G$ we can associate the function 
\begin{equation*}
f^G_{\vec{D}}:a\to 2
\end{equation*}
defined by $f(i)=0\leftrightarrow G\cap D_0^i\neq\emptyset$. Then two such sequences are equivalent $\leq p$ exactly if their associated functions are equal, for each choice of $G$.

\begin{defi}[Tameness]
$P$ is tame iff $P$ is pretame and for each $a\in M$ and $p\in P$ there is $q\leq p$ and $\alpha\in ORD(M)$ such that whenever $\vec{D}=\langle(D^i_0, D^i_1)\vert i\in a\rangle\in M$ is a sequence of predense $\leq q$ partitions, $\{r\,\vert\,\vec{D}$ is equivalent $\leq r$ to some $\vec{E}=\langle (E^i_0, E^i_1)\,\vert\, i\in a\rangle$ in $V^M_{\alpha}\}$ is dense below $q$.
\end{defi}

\begin{theo}\label{axioms}
Let $(M,\mathcal{C})$ be a model of MK. Then, if $G$ is $P$-generic over $(M,\mathcal{C})$ and $P$ is tame then $(M,\mathcal{C})[G]$ is a model of MK.
\end{theo}
\begin{proof}
Extensionality and Foundation follow because $M[G]$ is transitive (see Lemma \ref{basics} d) ; axiom 2 and 3 from Definitions \ref{names} and \ref{interpretation}. For Pairing, let $\sigma_1^G, \sigma_2^G$ be such that $\sigma_1$, $\sigma_2 \in \mathcal{N}^s$. Then the interpretation of the name $\sigma=\{\langle \sigma_1, 1^P\rangle, \langle\sigma_2, 1^P\rangle\}$ in the extension gives the desired $\sigma^G=\{\sigma_1^G, \sigma_2^G\}$. Infinity follows because $\omega$ exists in $(M,\mathcal{C})$ and the notion of $\omega$ is absolute to any model, $\omega\in (M,\mathcal{C})[G]$. Union follows as in the set-forcing case.

Replacement: This follows as in \cite{sybook} from the property of pretameness and we give the proof to make clear where the property of pretameness is needed: Suppose that $F:\sigma^G\to M[G]$. Then for each $\sigma_0$ of rank $<$ rank$\,\sigma$ the class $D(\sigma_0)=\{p\,\vert\,\text{for some }\tau, q\Vdash\sigma_0\in\sigma\to F(\sigma_0)=\tau\}$ is dense below $p$, for some $p\in G$ which forces that $F$ is a total function on $\sigma$. We now use pretameness to ``shrink'' this class to a set: so for each $q\leq p$ there is an $r\leq q$ and $\alpha\in Ord(M)$ such that $D_{\alpha}(\sigma_0)=\{ s\,\vert\, s\in V^{M}_{\alpha}\text{ and for some }\tau\text{ of rank }<\alpha, s\Vdash\sigma_0\in\sigma\to F(\sigma_0)=\tau\}$ is predense $\leq$ $r$ for each $\sigma_0$ of rank $<$ rank $\sigma$. Then it follows by genericity that there is a $q\in G$ and $\alpha\in Ord(M)$ such that $q\leq p$ and $D_{\alpha}(\sigma_0)$ is predense $\leq q$ for each $\sigma_0$ of rank $<$ rank$\,\sigma$. So let $\pi=\{\langle\tau, r\rangle\,\vert\text{ rank}\,\tau<\alpha, r\in V^{M}_{\alpha}, r\Vdash\tau\in\text{ ran}(F)\}$ and then it follows that ran$(F)=\pi^G\in M[G]$.

Power Set: This follows from tameness as shown in \cite{sybook}.

Class-Comprehension: 
 Let $\Gamma=\{\langle\sigma,p\rangle\in \mathcal{N}^s\times P\,\vert\, p\Vdash\varphi(\sigma, \Sigma_1,\ldots,\Sigma_n)\}$. Because of the Definability Lemma, we know that $\Gamma\in\mathcal{N}$. By Definition \ref{names} and \ref{interpretation}, $\Gamma^G=\{\sigma^G\,\vert\,\exists p\in G(\langle\sigma, p\rangle\in\Gamma)\}$ and we need to check that this equals the desired $Y=\{x\,\vert\,(\varphi(x, \Sigma_1^G,\ldots,\Sigma_n^G))^{(M,\mathcal{C})[G]}\}$. So let $\sigma^G\in\Gamma^G$. Then by the definition of $\Gamma^G$ we know that $p\Vdash\varphi(\sigma,\Sigma_1,\ldots,\Sigma_n)$ and because of the Truth Lemma it follows that $(M,\mathcal{C})[G]\models\varphi(\sigma^G,\Sigma_1^G,\ldots,\Sigma_n^G)$. For the converse, let $x\in Y$. By the Truth Lemma, $\exists p\in G(p\Vdash\varphi(\pi, \Sigma_1,\ldots,\Sigma_n)$, where $\pi$ is a name for $x$. By definition of $\Gamma$, $\langle\pi, p\rangle\in\Gamma$.

 
Global Choice: Let $<_M$ denote the well-order of $M$ and let $\sigma_x, \sigma_y$ be the least names for some $x,y\in M[G]$. As the names are elements of $M$, we may assume that $\sigma_x <_M \sigma_y$. So we define the relation $<_G$ in $M[G]$ using $M$ and $<_M$ as parameters, so that $x <_G y$ iff $\sigma_x <_M \sigma_y$ for the corresponding least names of $x$ and $y$. Let $R=\{(x,y)\,\vert\, x,y\in M[G]\text{ and }x <_G y\}$. Then by Class-Comprehension the class $R$ exists.
\end{proof}

\cite{sybook} gives us a simple sufficient condition for tameness that translates directly into the context of MK: 

\begin{defi}
For regular, uncountable $\kappa>\omega$, $P$ is $\kappa$-distributive if whenever $p\in P$ and $\langle D_i\,\vert\, i<\beta\rangle$ are dense classes, $\beta<\kappa$ then there is a $q\leq p$ meeting each $D_i$ ($p$ meets $D$ if $p\leq q\in D$ for some $q$).

$P$ is tame below $\kappa$ if the tameness conditions hold for $P$ with the added restriction that $Card(a) < \kappa$.
\end{defi}
\begin{lem}
If $P$ is $\kappa$-distributive then $P$ is tame below $\kappa$.
\end{lem}
\begin{proof}
Analogous to set-forcing \footnote{See \cite{sybook}, page 37.}.
\end{proof}

\section{Laver's Theorem}

In the following we will give an example which shows that a fundamental theorem that hold for set-forcing can be violated by tame class-forcings.

Laver's Theorem, as published in \cite{Lav}, shows that for a set-generic extension $V\subseteq V[G]$, $V\models ZFC$ with the forcing notion $P\in V$ and $G$ is $P$-generic over $V$, $V$ is definable in $V[G]$ from parameter $V_{\delta+1}$ (of $V$) and $\delta=\vert\,P\,\vert^+$ in $V[G]$. This result makes use of the fact that every such forcing extension has the approximation and cover properties as defined in \cite{ham2} and relies on certain results for such extensions. 

In general, the same does not hold for class-forcing. In fact there are class-forcings such that the ground model is not even second-order definable from set-parameters:

\begin{theo}
There is an MK-model $(M,\mathcal{C})$ and a first-order definable, tame class-forcing $\mathbb{P}$ with $G$ $\mathbb{P}$-generic over $(M,\mathcal{C})$ such that the ground model $M$ is not definable with set-parameters in the generic extension $(M,\mathcal{C})[G]$.
\end{theo}
\begin{proof}
We are starting from $L$. For every successor cardinal $\alpha$, let $P_\alpha$ be the forcing that adds one Cohen set to $\alpha$: $P_\alpha$ is the set of all functions $p$ such that
\begin{equation*}
dom(p)\subset\alpha,\quad\vert dom(p)\vert<\alpha,\quad ran(p)\subset\{0,1\}.
\end{equation*}
Let $P$ be the Easton product of the $P_\alpha$ for every successor $\alpha$: A condition $p\in P$ is a function $p\in L$ of the form $p=\langle p_\alpha : \alpha\text{ successor cardinal}\rangle\in\Pi_{\alpha\text{ succ.}} P_\alpha$ ($p$ is stronger then $q$ if and only if $p\supset q$) and $p$ has Easton support: for every inaccessible cardinal $\kappa$, $\vert\,\{\alpha<\kappa\,\vert\,p(\alpha)\neq\emptyset\}\,\vert<\kappa$. Then $P$ is the forcing which adds one Cohen set to every successor cardinal.   

Let $\mathbb{P}= P\times P=\Pi_{\alpha\text{ succ}} P_\alpha\times\Pi_{\alpha\text{ succ}} P_\alpha$ be the forcing that adds simultaneously two Cohen sets to every successor cardinal.\footnote{It follows by a standard argument that $\mathbb{P}$ is pretame (and indeed tame) over $(M,\mathcal{C})$, see \cite{sybook}.} Note that 
 $\Pi_{\alpha\text{ succ.}} P_\alpha\times\Pi_{\alpha\text{ succ.}} P_\alpha$ is isomorphic to $\Pi_{\alpha\text{ succ.}} P_{\alpha}\times P_{\alpha}$. Let $G$ be $\mathbb{P}$-generic. Then $G=\Pi_{\alpha\text{ succ.}} G_0(\alpha)\times G_1(\alpha)$ and we let $G_0=\Pi_{\alpha\text{ succ.}} G_0 (\alpha)$ and $G_1=\Pi_{\alpha\text{ succ.}} G_1 (\alpha)$ with $G_0,G_1$ $P$-generic over $L$. We consider the extension $L[G_0]\subseteq L[G_0][G_1]$ and we will show, that $L[G_0]$ is not definable in $L[G_0][G_1]$ from parameters in $L[G_0]$.
 
The reason that we cannot apply Laver's and Hamkins' results of \cite{Lav} to this extension is that it does not fulfill the $\delta$ approximation property\footnote{A pair of transitive classes $M\subseteq N$ satisfies the \emph{$\delta$ approximation property} (with $\delta\in Card^{N}$) if whenever $A\subseteq M$ is a set in $N$ and $A\cap a\in M$ for any $a\in M$ of size less than $\delta$ in $M$, then $A\in M$. For models of set theory equipped with classes, the pair $M\subseteq N$ satisfies the \emph{$\delta$ approximation property for classes} if whenever $A\subseteq M$ is a class of $N$ and $A\cap a\in M$ for any $a$ of size less than $\delta$ in $M$, then $A$ is a class of $M$.}: As the forcing adds a new set to every successor, the $\delta$ approximation property cannot hold at successor cardinals $\delta$: the added Cohen set is an element of the extension and a subset of the ground model and all of its $<\delta$ approximations are elements of the ground model but the whole set is not. 
 
 Note that the forcing is weakly homogeneous, i.e. for every $p,q\in\mathbb{P}$ there is an automorphism $\pi$ on $\mathbb{P}$ such that $\pi(p)$ is compatible with $q$. This is because every $P_{\alpha}$ is weakly homogeneous (let $\pi(p)\in P_\alpha$ such that $dom(\pi(p))=dom(p)$ and $\pi(p)(\lambda)= q(\lambda)$ if $\lambda\in dom(p)\cap dom(q)$ and $\pi(p)(\lambda)= p(\lambda)$ otherwise, then $\pi$ is order preserving and a bijection) and therefore also $P$ is weakly homogeneous (define $\pi$ componentwise 
 using the projection of $p$ to $p_\alpha$). Similar for $P\times P$. 
 
 
To show that $L[G_0]$ is not definable in $L[G_0][G_1]$ with parameters, assume to the contrary that there is a set-parameter $a_0$ such that $L[G_0]$ is definable by the second-order formula $\varphi(x, a_0)$ in $L[G_0][G_1]$ from $a_0$. Let $\alpha$ be such that $a_0\in L[G_0\!\!\upharpoonright\!\alpha, G_1\!\!\restriction\!\alpha]$. Now consider $a=G_0(\alpha^+)$, the Cohen set which is added to $\alpha^+$ in the first component of $\mathbb{P}$. $a$ is $P_{\alpha^+}$-generic over $L[G_0\!\!\restriction\!\alpha, G_1\!\!\restriction\!\alpha]$ and as $a$ is an element of $L[G_0]$ the formula $\varphi$ holds for $a$.
 So we also know that there is a condition $q\in G$ such that $q\Vdash\varphi(\dot{a}, a_0)$. 
 
Now we construct another generic $G^*=G^*_{0}\times G^*_{1}$ which produces the same extension but also an element for which $\varphi$ holds and which is not an element of $L[G_0]$. This new generic adds the same sets as $G$, but we switch $G_0$ and $G_1$ at $\alpha^+$ so that the set added by $G_1(\alpha^+)$ is now added in the new first component $G_0^*$. However we have to make sure that the new generic respects $q$ so that $\varphi$ is again forced in the extension. We achieve this by fixing the generic $G$ on the length of $q(\alpha^+)$ (we can assume that the length is the same on $G_0$ and $G_1$). 

It follows that $q\in G_{0}^{*}\times G_1^*$ and because of weakly homogeneity $G_{0}^{*}\times G_1^*$ is generic and $L[G_0][G_1]=L[G_0^*][G_1^*]$. Because of the construction of $G^*$, the formula $\varphi(x, a_0)$ holds for the set $b=G_0^*(\alpha^+)$ but $b$ is not an element of $L[G_0]$. That is a contradiction!

\end{proof}

We have seen that there are different ways of approaching class-forcing, namely on the one hand as definable from a class parameter $A$ in a ZFC model $(M, A)$ and on the other hand in the context of an MK model $(M,\mathcal{C})$. That presents us with three notions of genericity: set-genericity, $A$-definable class genericity and class-genericity. One of the questions that arises now is in which way we can define the next step in this ``hierarchy'' of genericity. To answer this question, Sy Friedman and the author of this paper are currently working on so-called hyperclass-forcings in a variant of MK, i.e. forcings in which the conditions are classes (see \cite{hyper}). We will show in which context such forcings are definable and which application they have to class-theory.



\end{document}